\documentclass[12pt]{amsart}
\usepackage{amsmath}
\usepackage{amssymb}
\usepackage{amsthm}
\usepackage{amscd}
\usepackage{mathrsfs}

\usepackage{color}
\usepackage{hyperref}

\usepackage[latin1]{inputenc}
\usepackage{mathpazo}
\usepackage[scaled=.95]{helvet}
\usepackage{courier}

\swapnumbers
\headsep=1cm \numberwithin{equation}{section}
\usepackage[colorinlistoftodos]{todonotes}

\newtheorem{theorem}{Theorem}[section]
\newtheorem{lemma}[theorem]{Lemma}
\newtheorem{proposition}[theorem]{Proposition}
\newtheorem{corollary}[theorem]{Corollary}

\theoremstyle{definition}
\newtheorem{definition}[theorem]{Definition}
\newtheorem{remark}[theorem]{Remark}
\newtheorem{remark and definition}[theorem]{Remark and Definition}
\newtheorem{remark and notation}[theorem]{Remark and Notation}

\newtheorem{question}[theorem]{Question}

\renewcommand{\theacknowledgement}{}

\usepackage{ulem,xpatch}

\newcommand\Proj{\operatorname{Proj}}
\newcommand\Spec{\operatorname{Spec}}
\newcommand\Hom{\operatorname{Hom}}
\newcommand\Ext{\operatorname{Ext}}
\newcommand\Tor{\operatorname{Tor}}
\newcommand\height{\operatorname{height}}
\newcommand\depth{\operatorname{depth}}
\newcommand\codim{\operatorname{codim}}
\newcommand\Sym{\operatorname{Sym}}
\newcommand\beg{\operatorname{beg}}
\newcommand\reg{\operatorname{reg}}
\newcommand\Ker{\operatorname{\Ker}}
\newcommand\Coker{\operatorname{Coker}}
\newcommand\e{\operatorname{e}}
\newcommand\pd{\operatorname{pd}}
\newcommand\id{\operatorname{id}}
\newcommand\Gid{\operatorname{Gid}}
\newcommand\cdim{\operatorname{CI-dim}}
\newcommand\gdim{\operatorname{G-dim}}
\newcommand\Supp{\operatorname{Supp}}
\newcommand\Rad{\operatorname{Rad}}
\newcommand\Ass{\operatorname{Ass}}
\newcommand\Ann{\operatorname{Ann}}
\newcommand\Nor{\operatorname{Nor}}
\newcommand\Sec{\operatorname{Sec}}
\newcommand\NZD{\operatorname{NZD}}
\newcommand\Sing{\operatorname{Sing}}
\newcommand\rank{\operatorname{rank}}
\newcommand{\xx}{\underline x}
\newcommand{\yy}{\underline y}
\newcommand{\qism}{\stackrel{\sim}{\longrightarrow}}
\newcommand{\lqism}{\stackrel{\sim}{\longleftarrow}}

\date{}

\begin{document}

\title{On Coefficient Modules  of Arbitrary Modules}
\thanks{{Work
partially supported by CNPq-Brazil-Grants 309316/2011-1,
and FAPESP Grant 2012/20304-1}. 
}
\maketitle


\centerline{}

\centerline{\bf {V. H. Jorge P\'{e}rez}}

\centerline{{\small vhjperez@icmc.usp.br}}

\centerline{Departamento de Matem\'{a}tica}

\centerline{ICMC - Instituto de Ci\^{e}ncias da Computa\c{c}\~{a}o e Matem\'{a}tica}

\centerline{USP - Universidade de S\~{a}o Paulo}

\centerline{}

\centerline{\bf {M. D. Ferrari}}

\centerline{{\small mdsilva@uem.br}}

\centerline{Departamento de Matem\'{a}tica}

\centerline{CCE - Centro de Ci\^{e}ncias Exatas}

\centerline{UEM - Universidade Estadual de Maring\'{a}}

\centerline{}

\centerline{\bf {P. H. Lima}}

\centerline{{\small apoliano108@gmail.com}}

\centerline{Departamento de Matem\'{a}tica}

\centerline{CCET - Centro de Ci\^encias Exatas e Tecnologias}

\centerline{UFMA - Universidade Federal do Maranh\~ao}








\newtheorem{rem}[theorem]{\quad Remark}


\begin{abstract}
Let $(R, \mathfrak{m})$ be a $d$-dimensional Noetherian local ring that is formally equidimensional, and let $M$ be an arbitrary $R$-submodule of the free module $F = R^p$ with an analytic spread $s:=s(M)$. In this work, inspired by Herzog-Puthenpurakal-Verma in \cite{herzog}, we show the existence of an unique largest $R$-module $M_{k}$ with $\ell_R(M_{k}/M)<\infty$ and
$M\subseteq M_{s}\subseteq\cdots\subseteq M_{1}\subseteq
M_{0}\subseteq q(M),$ such that $\deg(P_{M_{k}/M}(n))<s-k,$ where $q(M)$ is the relative integral closure of $M,$ defined by $q(M):=\overline{M}\cap M^{sat},$ where $M^{sat}=\cup_{n\geq
1}(M:_F\mathfrak{m}^n)$ is the saturation of $M$. We also provide a structure theorem for these modules. Furthermore, we establish the existence of coefficient modules between $I(M)M$ and $M$, where $I(M)$ denotes the $0$-th Fitting ideal of $F/M$, and discuss their structural properties. Finally, we present some applications and discuss some properties.
\end{abstract}



\section{Introduction}
Let $(R,\mathfrak{m})$ be a commutative Noetherian local ring of Krull dimension $d$, and let $I$ be an $\mathfrak{m}$-primary ideal of $R$ with its integral closure denoted by $\overline{I}$. The classical result known as the Hilbert-Samuel theorem says that for sufficiently large integer $n$, the length of $R/I^n$  is a polynomial in variable $n$ of degree $d$. This polynomial is well known as the Hilbert-Samuel polynomial of $I$ and can be expressed as:
$$
\sum^{d}_{i=0}(-1)^i{\rm e}_i(I)\left(
\begin{array}{c}
{n+d-i-1}\cr d-i
\end{array}
\right),
$$
\noindent where ${\rm e}_i(I)$ is called the $i$-th Hilbert-Samuel coefficient of $I$ for $i=1,\ldots,d$ and ${\rm e}_0(I)$ is the Hilbert-Samuel multiplicity of $I$. In this context, we have the following classical result known as Rees multiplicity theorem (see \cite{northcott} and Kirby-Rees \cite{rees}): If $R$ is formally equidimensional, then the integral closure $\overline{I}$ of $I$ can be characterized as the unique largest ideal of $R$ containing $I$ and having the same Hilbert-Samuel multiplicity as $I$, i.e., ${\rm e}_0(I)={\rm e}_0(\overline{I})$. It should be noted that this result is extensively used in Algebraic Geometry, Singularity Theory and Commutative Algebra (see \cite{kleiman}). Given the importance of this result, it is natural to ask the following question:

\begin{question}\label{Question1}
Let $(R,\mathfrak{m})$ be a Noetherian formally equidimensional local ring. Is there a chain of ideals between $I$ and $\overline{I}$ with a property similar to Rees multiplicity theorem? If so, there exists a colon structure for these ideals?
\end{question}

Motivated by this question or Rees' theorem, Shah establishes in \cite{shah}  the existence of a chain of unique largest ideals $I_k$, with $0\leq k\leq d$, between the ideals $I$ and $\overline{I}$. Each $I_k$ ideal, called Coefficient Ideal, have the property of preserving the first $k+1$ Hilbert-Samuel coefficients of $I$. In other words, there exists ideals $I_k$ satisfying the following inclusions: $I \subset I_d \subset \cdots \subset I_1 \subset I_0 = \overline{I}$ with ${\rm e}_i(I)={\rm e}_i(I_{k})$, for $i=0,\ldots,k$. The $k$-th ideal coefficient $I_k$ represents the $k$-th step in this chain. Moreover, Shah determined a specific colon ideal structure for each coefficient ideal $I_k$ of $I$. He showed that, for $k=1,\ldots,d$, the ideal $I_k$ is determined by $I_k = (I^{n_0+1}:x_1,\ldots,x_k)$, where $n_0\geq1$ is a fixed integer and $(x_1,\ldots,x_d)$ is a fixed minimal reduction of $I^{n_0}$, for some integer $n_0$. Additionally, Shah proved that the Ratliff-Rush closure of this same ideal, denoted as $\widetilde{I}:=\cup_{n\geq1}(I^{n+1}:I^n)$, coincides with the $d$-th ideal coefficient, i.e., $I_d = \widetilde{I}$. 

The importance of the Ratliff-Rush closure and coefficient ideals cannot be underestimated. For instance, the Ratliff-Rush filtration $\{(\widetilde{I})^n\}_{n\in\mathbb{N}}$ has proven to be a very useful tool for studying numerical invariants of the associated graded ring $G_I(R)=\oplus_{n\geq0}I^n/I^{n+1}$ (see Rossi-Valla, \cite{rossi2}, and the references therein). Rossi-Valla used the Ratliff-Rush closure of an ideal crucially in their proof of the Sally conjecture. Blancafort also employed the Ratliff-Rush closure of an ideal and the Hilbert-Samuel function of a good filtration to generalize the works of Huckaba-Marley in \cite{huckaba} and Rossi-Valla in \cite{rossi}.

On the other hand, coefficient ideals also hold their importance. For instance, Ciuperc\u{a} described the $S_2$-ification of the extended Rees algebra of a primary $\mathfrak{m}$-ideal $I$ through the first coefficient ideal of all powers of $I$ in \cite{ciuperca}. Furthermore, assuming that $I$ is a $0$-dimensional monomial ideal in characteristic zero, Polini-Ulrich-Vitulli used the fact that the kernel of an ideal can be related to components of the graded canonical module of the extended Rees algebra to prove that the kernel of $I$ coincides with the kernel of any ideal between $I$ and the first coefficient ideal $I_{1}^R$ of $I$ in \cite{polini}. Moreover, the properties of the first coefficient ideals were extensively described by Heinzer-Johnston-Lantz, Heinzer-Johnston-Lantz-Shah, Heinzer-Lantz, and Shah in \cite{heinzer, heinzer2, heinzer3, shah}, respectively.

We can observe that the above results were obtained for $\mathfrak{m}$-primary ideals, but Question \ref{Question1} was also answered for arbitrary ideals: Let $J\subset I$ be ideals of $R$ with $\ell_R(I/J)<\infty$. Amao in \cite{amao} and Rees in \cite{rees2} show that the numerical function $h_{I/J}(n):=\ell_R(I^n/J^n)$ is a polynomial function $P_{I/J}(n)$ of degree at most $d$. We refer to $h_{I/J}(n)$ and $P_{I/J}(n)$ as the Rees-Amao function and the Rees-Amao polynomial of the pair $(J,I)$, respectively. The Rees multiplicity theorem has been generalized for ideals that are not $\mathfrak{m}$-primary by several authors. For instance, Rees showed in \cite{rees2} that if $(R,\mathfrak{m})$ is formally equidimensional, then $J$ is a reduction of $I$ if and only if $\deg(P_{I/J}(n))<d$.

Analogously to the results obtained by Shah in \cite{shah} for the case of $\mathfrak{m}$-primary ideals, Herzog-Puthenpurakal-Verma proved in \cite[Theorem 4.4]{herzog} that coefficient ideals always exist for arbitrary ideals. More specifically, if $(R,\mathfrak{m})$ is a formally equidimensional local ring and $I$ is an ideal of $R$, then for each $k=1,\ldots,s$, there exists a unique largest ideal $I_{k}$ with $\ell_R(I_{k}/I)<\infty$, and 
\[
I\subseteq I_{s}\subseteq I_{s-1}\subseteq\cdots\subseteq I_{1}\subseteq q(I), \mbox{ such that } \deg\left(P_{I_{k}/I}(n)\right)<s-k
\]
where $q(I)$ is the relative integral closure of $I$, and $s := \dim F(I)=\dim(\oplus_{n\geq 0}I^n/\mathfrak{m}I^n)$ is the analytic spread  of $I$, where $F(I)$ is the fiber cone of $I$,

It should be noted that Herzog-Puthenpurakal-Verma showed the existence of ideal coefficients for arbitrary ideals, but they did not demonstrated the colon structure of these ideals. However, in \cite{CPF}, the authors determined the colon structure of each ideal coefficient found by Herzog-Puthenpurakal-Verma. 

Finally, as we can see, the ideal coefficients have been extensively explored by many researchers. Therefore, it is natural to wonder if these results can be extended to modules. To introduce this context, we recall some well-known facts.

Let $M$ be a finitely generated $R$-module contained in a free $R$-module $F$ of positive rank $p>0$ with finite colength, denoted as $\ell_{R}(F/M)<\infty$. Using the notation described in Section \ref{section2}, Buchsbaum-Rim showed that for sufficiently large values of $n$, the length $\ell_R(F^n/M^n)$ can be expressed as a polynomial in $n$ of degree $d+p-1$, given by
$$P_M^F(n)=\sum\limits_{i=0}^{d+p-1}{(-1)}^i{\rm e}_i^F(M)\left(\begin{array}{c}n+d+p-i-2\cr d+p-i-1\end{array}\right),$$
where ${\rm e}_i^F(M)$ represents the $i$-th Buchsbaum-Rim coefficient of $M$ in $F$ for $i=0,\dots,d+p-1$. In particular, ${\rm e}_0^F(M)$ is called the Buchsbaum-Rim multiplicity of $M$ in $F$ (for more detailed see \cite{buchsbaum}). 

Similar to the case of ideals, there is also the Rees' multiplicity theorem for modules. If $R$ is formally equidimensional, then the integral closure $\overline{M}$ of $M$ in $F$ can be characterized as the largest submodule of $R$ that contains $M$ and has the same Buchsbaum-Rim multiplicity as $M$, i.e., ${\rm e}_0^F(M)={\rm e}_0^F(\overline{M})$ \cite{rees}. The Buchsbaum-Rim multiplicity has been extensively studied by numerous authors in the fields of Commutative Algebra, Algebraic Geometry, and Singularity Theory (e.g., \cite{buchsbaum, rees, Katz, kleiman}, \cite[p. 317]{huneke} and  \cite[p. 418, Remark 8.9]{vasconcelos}). Given the importance of this result, it is natural to ask the following:

\begin{question}
Let $(R,\mathfrak{m})$ be a Noetherian formally equidimensional ring, and let $M$ be an $R$-submodule of $F$. Is there a chain of modules between $M$ and $\overline{M}$ with a property similar to Rees's multiplicity theorem? If so, there exists a colon structure for these modules?
\end{question}

We can observe that this question is less explored compared to the ideal case, but there have already been some partial answers. For instance, Liu \cite{jung} provides the following result: If $R$ is a domain formally equidimensional ring, and $M$ is a finitely generated $R$-module such that $F/M$ has finite length, then there exists a unique chain 
$$M \subseteq M^F_{\{d+p-1\}} \subseteq \cdots \subseteq M^F_{\{1\}} \subseteq M^F_{\{0\}} \subseteq \overline{M},$$
such that ${\rm e}_i^F(M) = {\rm e}_i^F(M^F_{\lbrace k\rbrace})$, for all $i = 0,\ldots,k$. The module $M^F_{\lbrace k\rbrace}$ is called the Coefficient Module of $M$. On the other hand, in \cite{Perez-Ferrari1}, the P\'{e}rez and Ferrari showed the same result as Liu, but without assuming that $R$ is a domain. Furthermore, they provided the colon structure of these modules, which constitutes a generalization of one of the main results given by Shah (\cite[Theorem 2 and 3]{shah}).

The main aim of this paper is to establish a generalization of a result originally presented by Herzog-Puthenpurakal-Verma in \cite[Theorem 4.4]{herzog}. Precisely, we show that exists an unique largest module $M_{k}$ with $\ell_R(M_{k}/M)<\infty$ and
$M\subseteq M_{s}\subseteq\cdots\subseteq M_{1}\subseteq
M_{0}\subseteq q(M),$ such that $\deg(P_{M_{k}/M}(n))<s-k,$ where $q(M)$ is the relative integral closure of $M,$ defined by $q(M):=\overline{M}\cap M^{sat},$ where $M^{sat}=\cup_{n\geq
1}(M:_F\mathfrak{m}^n)$ is the saturation of $M$ and $s$ is the analytic spread of $M$ (see Theorem \ref{teoralgcoefarb}). Additionally, we also provide the colon structure of these coefficient modules (see Theorems \ref{teorcoefestruturaparmod},\ref{teorcoefestruturafinaparmod}). Furthermore, our results also generalizes the results given by Liu in \cite{jung} and Perez-Ferrari in \cite{Perez-Ferrari1}.

Now let $I(M)$ denote the $0$-th Fitting
ideal $\text{Fitt}_0(F/M)$ of $F/M$. In \cite{Katz2}, Katz and Kodiyalam defined the associated graded ring with respect to the module $M$ to be the
$R/I(M)$-algebra
$$G_{I(M)}(\mathscr{R}(M)):=\frac{\mathscr{R}(M)}{I(M)\mathscr{R}(M)} = \frac{R}{I(M)} \oplus \frac{M}{I(M)M} \oplus \frac{M^2}{I(M)M^2} \oplus \cdots.$$ In the case $p=1$, that is, if $F=R$ is a ring and $M=I$ is an ideal of $R$, then $I(M)=I$ and $G_{I(M)}(\mathscr{R}(M))=G_I(R)$, the associated graded ring of
the ideal $I$ of $R$.
Another aim of this paper is to
give a new result that shows a way to control the height of the associated prime ideals of $G_{I(M)}(\mathscr{R}(M))$ (see Theorem \ref{height}), providing a generalization of one main results of Shah \cite[Theorem 4]{shah}. To do that, we establish the existence of another coefficient modules between the $R$-modules $I(M)M$ and $M$, denoted by $M_{[k]}$ and which we call the coefficient module of the associated graded ring $G_{I(M)}(\mathscr{R}(M))$ (see Section \ref{CoefAss}). Additionally, we give a colon structure of these coefficient modules in Theorem \ref{structure1}.

The organization of this paper is as follows: In section 2, we introduce the notation, definitions, and some necessary known results that will be used. 

In Sections 3 and 4, we show the main results of this work. In Section 3, we start by proving the existence of coefficient modules between $M$ and $q(M)$. Additionally, we provide an algebraic characterization of the coefficient modules in terms of colon modules. In section 4, we show the existence of coefficient modules between $I(M)M$ and $M$, and give structural properties for them.

In the last section, we provide some applications and properties by making use of the results presented in Sections 3 and 4.

\section{Setup and Background}

Let $(R,\mathfrak{m})$ be a Noetherian local ring, $M$ a finitely generated $R$-module which is contained in a free $R$-module $F$ of positive rank. The embedding $M\subseteq F$ induces the graded $R$-algebra homomorphism between the symmetric algebras
$${\rm Sym}(i):{\rm Sym}_R(M)\longrightarrow {\rm Sym}_R(F)$$
of $M$ and $F$. As $F$ is a free $R$-module, the symmetric algebra $S(F):={\rm Sym}_R(F)$ of $F$ concides with the polynomial ring $S=R[t_1,t_2,\dots,t_p]$ over $R$, where $p={\rm rank}_R(F)>0$. The authors Simis-Ulrich-Vasconcelos in \cite{vasconcelos}, defined the Rees algebra $\mathscr{R}(M)$ of the module $M$ as the image of the induced homomorphism;
$$\mathscr{R}(M) = {\rm Im}\left({\rm Sym}(i)\right) = \bigoplus_{n\geq0} M^n \subseteq S(F)= R[t_1,t_2,\ldots,t_p] $$
where $M^n:=[\mathscr{R}(M)]_n$ stand for the homogeneous component of $\mathscr{R}(M)$ of degree $n$, i. e., is the $n$-th power of the image of $M$ in $\mathscr{R}(M)$. In particular, $M=[\mathscr{R}(M)]_1$ is an $R$-submodule of $\mathscr{R}(M)$. We denote $I_M$ the ideal of $S$ generated by $[\mathscr{R}(M)]_1$. So the fiber cone of ideal ${I_M}$ is defined as $\mathscr{F}({I_M}):=\oplus_{n\geq 0} I_M^n/\mathscr{M} I_M^n$, where $ \mathscr{M}=\mathfrak{m}\oplus \bigoplus_{n>0}F^n$ is the only maximal homogeneous ideal of $S.$
 
As in the ideal case, we may define the fiber cone (or special fiber) of an  $R$-module $M$ as following $\mathscr{F}(M):={\mathscr{R}(M)}/\mathfrak{m}{\mathscr{R}(M)}.$ The analytic spread of $M$ is defined by $s:=s(M)=\dim(\mathscr{F}(M))$.

Throughout the paper the $R$-modules will be contained in a free $R$-module $F$ of finite rank $p> 0.$

The next result gives a characterization for a sequence to be a minimal reduction of powers of a module.

\subsection{Some properties of reduction}\label{section2}

An $R$-submodule $N$ of $M$ is called {\it reduction} of $M$ if there exists an $n_0>0$ such that $M^{n+1}=NM^n \mbox{ for all } n\geq n_0.$ A reduction $N\subseteq M$ is said to be a minimal reduction of $M$ if $N$ properly contains no further reductions of $M$. If the $R$-submodule $N$ is a reduction of the $R$-submodule $M$, then there exists at least one $R$-submodule $L$ in $N$ such that $L$ is a minimal reduction of $M$. This last statement follows exactly from the same deduction as the proof of Huneke-Swanson in \cite[Theorem 8.3.5]{huneke}.

Let us recall the definition of the integral closure of modules. For every integer $n\geq0$, we define the integral closure
$$\overline{M^n}=\left[\overline{\mathscr{R}(M)}^{{\rm Sym}(F)}\right]_n\subseteq S_n(F)=F^n$$
of $M^n$ to be the $n$-th homogeneous component of the integral closure $\overline{\mathscr{R}(M)}^{{\rm Sym}(F)}$ of $\mathscr{R}(M)$ in ${\rm Sym}(F)$. In
other words, $\overline{M^n}$ is the integral closure of the ideal $(M{\rm Sym}(F))^n$ of degree $n.$ In particular, $\overline{M}=(\overline{M{\rm Sym}(F)})_1\subseteq F.$ Hence $\overline{M}$ consists of the element $x\in F$ which satisfies the
integral equation $x^n+c_1x^{n-1}+\cdots+c_{n-1}x+c_n=0$ in ${\rm Sym}(F)$ where $n>0$ and $c_i \in M^i$ for every $1\leq i\leq n.$


\noindent {\bf Saturation and reduction}: Let $(R,\mathfrak{m})$ be a Noetherian local ring, $M$ a finitely generated $R$-submodule  of $F$. The saturation $M^{sat}$ of $M$ with respect to $\mathfrak{m}$  is the $R$-submodule of $F$ defined by
$$
M^{sat}=\bigcup_{n\in N}\left(M:_F\mathfrak{m}^{n}\right)
=\left\{x\in F\ |\ x\mathfrak{m}^n\subseteq M,\ \textrm{for some}\
n\in \mathbb{N} \right\}
$$
and the relative integral closure $q(M)$ of $M$ is defined by $q(M):=\overline{M} \cap M^{{sat}}$.

The saturation and the relative integral closure are $R$-modules which contains $M$. 

Since $M\subseteq (M:_F\mathfrak{m}) \subseteq (M:_F\mathfrak{m}^{2}) \subseteq (M:_F\mathfrak{m}^{3})\subseteq \cdots $ and since $F$ is a Noetherian $R$-module, there exists an integer $k$ such that $(M:_F\mathfrak{m}^{k+1})=(M:_F\mathfrak{m}^{k})$. This implies that $\ell_R(M^{{sat}}/M)<\infty$. 

\begin{lemma}\label{propproperties} Let $(R,\mathfrak{m})$ be a Noetherian local ring, $M$ a finitely generated $R$-module which is contained in $F$. Then, the relative integral closure  $q(M)$ of $M$ satisfies the following properties:
\begin{enumerate}
\item[(i)] $M$ is a reduction of $q(M)$ and $\ell_R(q(M)/M)<\infty$.
\item[(ii)] If $M$ is a reduction of $N$ with $\ell_R(N/M)<\infty$ then $N\subseteq q(M)$.
\item[(iii)] If $N\subseteq M$ then $q(N)\subseteq q(M)$.
\item[(iv)] $q(q(M))=q(M)$.
\end{enumerate}
\end{lemma}

\begin{proof}
(i) Since $M\subseteq q(M)\subseteq \overline{M}$, we get $M$ is a reduction of $q(M)$. Also, we have 
$\ell_R(q(M)/M)<\ell_R(M^{sat}/M)<\infty.$

(ii) Since $M$ is a reduction of $N$ we get $q(N)=\overline{N}\cap N^{sat}=\overline{M}\cap N^{sat}$. As $\ell_R(N/M)<\infty$, there is $k\in \mathbb{N}$ such that $\mathfrak{m}^kN\subseteq M$. Therefore $N^{sat}\subseteq M^{sat}$ and $N\subseteq q(N)\subseteq \overline{M}\cap M^{sat}=q(M)$.

(iii) Clear. 

 (iv) If $q(M)$ is a reduction of $K$ with $\ell_R(K/q(M))<\infty$ then $M$ is also a reduction of $K$ and $\ell_R(K/M)<\infty$. So by $(ii)$ we have $K\subseteq q(M)$. By taking $K=q(q(M))$ we get the desired result.

\end{proof}

\begin{lemma}\label{lemaferramentas} {\rm (\cite[Lemma 2.3]{Perez-Ferrari1})} Let $M$ be an $R$-module  and let $a_1,\ldots,a_s$ be elements of $M^{n_0}$ for some fixed $n_0>0$. We denote by $\overline{a_i}:=a_i+\mathfrak{m}M^{n_0}$, for $i=1,\ldots,s$. Then
$$\ell_{\frac{R}{\mathfrak{m}}}\left(\frac{\mathscr{F}(M)}{\left(\overline{a}_1,\ldots,\overline{a}_s\right)}\right)<\infty
\Longleftrightarrow\ 
\left(a_1,\ldots,a_s\right)\ \mbox{is a reduction of}\ M^{n_0}.$$
\end{lemma}

The next theorem, when the residue field is infinite, shows the existence of minimal reductions whose minimal generating sets have cardinality equal to the analytic spread of the $R$-module $M$. As consequence, the following corollary shows a condition, using minimal reduction, to an element belongs to the integral closure.

\begin{theorem} {\rm (\cite[Corollary 16.4.7]{huneke})} Let $(R,\mathfrak{m})$ be a Noetherian local ring with infinite residue field. Then all $R$-module $M$ with analytic spread $s(M)$ admits a reduction generated by $s(M)$ elements. Furthermore, this reduction is minimal.
\end{theorem}


\begin{corollary} \label{coropropoeliminandoelred2} {\rm (\cite[Corollary 2.6]{Perez-Ferrari1})} Let $R$ be formally equidimensional Noetherian ring, $M$ an $R$-module and $s:=s(M)$. Let $x_1,\ldots,x_s$ form a minimal reduction of $M^{n_0}$ for some $n_0\geq 1$. Suppose $yx_i\in M^{n_0+1}$ for some $y\in S(F)$ and some $i\in \{1,\ldots,s\}.$ Then $y\in\overline{M}$.
\end{corollary}
Some general references on the subjects of this Section 2 are authored by Simis-Ulrich-Vasconcelos in \cite{vasconcelos} and some of these results were proved in Huneke-Swanson in \cite{huneke}.
\section{The modules coefficients between $M$ and $\overline{M}$}\label{CoeffMod}

In this section, we show the main results of this work. We begin by showing the existence of coefficient modules between $M$ and $q(M)$. Additionally, we provide an algebraic characterization of the coefficient modules in terms of colon modules. The results obtained in this section generalize the findings of Herzog-Puthenpurakal-Verma (\cite[Theorem 4.4]{herzog}), Liu (\cite{jung}), and Perez-Ferrari (\cite{Perez-Ferrari1}).

\begin{theorem}[Theorem of Existence for Arbitrary Modules]\label{teoralgcoefarb} Let $(R,\mathfrak{m})$ be a formally equidimensional and $d$-dimensional local ring with infinite residue field. Let $M$ be a finitely generated $R$-submodule of $F$ and  with analytic spread $s:=s(M)\geq 1$ and let $k \in \{1,\ldots,s\}$. Then there exist unique largest 
$R$-submodules $M_{k}$ of $F$ with $\ell_R(M_{k}/M)<\infty$ and
$$M\subseteq M_{s}\subseteq\cdots\subseteq M_{1}\subseteq 
M_{0}=q(M)$$
such that ${\rm deg} (P_{M_{k}/M}(n))<s-k$ for $k=1,\ldots,s$.
\end{theorem}

\begin{proof} Fix $0\leq k\leq s$. Consider the sets

$$V_{k}=\left\{ L\subseteq F \ | \ \ M\subseteq L,\ 
\ell_R\left(\frac{L}{M}\right)<\infty,\ \mbox{deg}\left(P_{\frac{L}{M}}(n)\right)<s-k\right\}$$

If $L\in V_{k}$, then $\mbox{deg}(P_{L/M})<s-1$ and $\ell_R(L/M)<\infty$. Hence, by Huneke-Swanson's theorem \cite[Theorem 6.5.6 (3)]{huneke}, $M$ is a reduction of $L$ and by  Lemma \ref{propproperties} item $(ii)$, $L\subseteq q(M)$.
Since $M\in V_{k}$ and $F$ is a Noetherian $R$-module, $V_{k}$ has a maximal element, say $N$.

We show that $N$ is unique. Let $L\in V_{k}$ and $x\in L$. Since $M\subseteq (M,x)\subseteq L$,
$$\ell_R\left(\frac{{(M,x)}^n}{M^n}\right)\leq \ell_R\left(\frac{L^n}{M^n}\right)<\infty,$$
for all large $n$. Hence
$$\mbox{deg}\left(P_{\frac{(M,x)}{M}}\right)\leq \mbox{deg}\left(P_{\frac{L}{M}}\right)<s-k.$$
By \cite[Theorem 6.5.6]{huneke}, $M$ is a reduction of $(M,x)$, i.e., $M{(M,x)}^r={(M,x)}^{r+1}$ for some $r$. Then $x^{r+1}\in M{(M,x)}^{r}\subseteq N{(N,x)}^r$. Hence $N{(N,x)}^r={(N,x)}^{r+1}$. It follows that, for all $n\geq r$, $N^{n-r}{(N,x)}^r={(N,x)}^n$.

Because $\ell_R((M,x)/M)<\infty$, there exists some $t>0$ such that $\left(x,M\right){\mathfrak{m}}^t\subseteq M\subseteq N$. Then $x{\mathfrak{m}}^t\subseteq M\subseteq N$ and ${\mathfrak{m}}^tM\subseteq M\subseteq N$. Hence ${\mathfrak{m}}^tM\subseteq M\subseteq\left(N,x\right)$, so that $\ell_R\left((N,x)/M\right)<\infty$. We have
$$
\begin{array}{lll}
\displaystyle\ell_R\left(\frac{{\left(N,x\right)}^n}{M^n}\right) & =  & \displaystyle\ell_R\left(\frac{{(N,x)}^tN^{n-t}}{M^n}\right) \\
& =  &  \displaystyle\ell_R\left(\frac{\left(N^n,N^{n-1}x,\dots,N^{n-t}x^t\right)}{M^n}\right)\\
& \leq & \displaystyle\sum^t_{i=1}\ell_R\left(\frac{N^{n-i}x^i+N^n}{M^n}\right) \\
\\
& \leq & \displaystyle\sum^t_{i=1}\left[\ell_R\left(\frac{N^{n-i}x^i+M^n}{M^n}\right)+\ell_R\left(\frac{N^n}{M^n}\right)\right] \\
\\
& \leq & \displaystyle\sum^t_{i=1}\left[\ell_R\left(\frac{M^{n-i}x^i+M^n}{M^n}\right)+\ell_R\left(\frac{N^n}{M^n}\right)+\ell_R\left(\frac{N^{n-i}x^i+M^n}{M^{n-i}x^i+M^n}\right)\right] \\
& \leq & \displaystyle\sum^t_{i=1}\left[\ell_R\left(\frac{{\left(M,x\right)}^n}{M^n}\right)+\ell_R\left(\frac{N^n}{M^n}\right)+\ell_R\left(\frac{N^{n-i}}{M^{n-i}}\right)\right].
\end{array}
$$
Hence $\mbox{deg}(P_{(N,x)/M})<s-k.$ Therefore $(N,x)\subseteq V_{k}$. By maximality of $N$, we conclude that $x\in N$. Hence $L\subseteq N$ and therefore $N$ is unique and it will be denoted by $M^F_{k}$.
\end{proof}

\begin{definition}
The $R$-submodules $M^F_{k}$ of $F$, for $k=1,\ldots, s$, which appears in the Theorem \ref{teoralgcoefarb}, are called the $k$-th Coefficient Modules of $M$ in $F$. 
\end{definition}

\begin{theorem}\label{teorcoefestruturaparmod}{\rm [Structure Theorem for Coefficient Modules]}
Let $(R,\mathfrak{m})$ be a formally equidimensional and $d$-dimensional local ring with infinite residue field. Let $M$ be a finitely generated $R$-submodule of $F$ with analytic spread $s:=s(M)\geq 1$. Let $k\in \{1,\dots,s\}$. Then
$$
M^F_{k}=\bigcup_{n, \ \overline{x}}\left(M^{n_0+1}:_Fx_1,\ldots,x_k\right)\cap M^{sat},$$
for some fixed $n_0\geq 1$ and $\overline{x}=( x_1,\ldots,x_k,\dots,x_s)$ is a minimal reduction of $M^{n_0}$.
\end{theorem}

\begin{proof}

 Fix $1\leq k\leq s$ and let $y\in M^F_{k}$.  Then $M\subseteq(M,y)\subseteq M^F_{k}$, and hence
$$
\ell_R\left(\frac{{\left(M,y\right)}^n}{M^n}\right)\leq P_{\frac{M^F_{k}}{M}}(n),
$$
where, by definition of coefficient module we have
$\deg(P_{M^F_{k}/M}(n))<s-k.$

In particular, since $(M,y)M^{n-1}\subseteq{(M,y)}^n$,
\begin{equation}\label{eqteoralgcoefestrutura1}
\ell_R\left(\frac{{\left(M,y\right)}M^{n-1}}{M^n}\right)
\leq\ell_R\left(\frac{{\left(M,y\right)}^n}{M^n}\right)\leq
P_{\frac{M^F_{k}}{M}}(n).
\end{equation}

Set
$$
L=\bigoplus_{n\geq0}L_n,\quad\textrm{ where }\quad L_n:=\frac{\left(M,y\right)M^n}{M^{n+1}}.
$$
Clearly $L$ has the structure of an $\mathscr{R}(M)$-module. By equation (\ref{eqteoralgcoefestrutura1}) we have that
$\dim(L)\leq s-k,$ that is,
\begin{equation}\label{eqteoralgcoefestrutura2}\dim(L)=\dim\left(\frac{\mathscr{R}(M)}{\left(0:_{\mathscr{R}(M)}L\right)}\right)\leq s-k.
\end{equation}

We can easily prove that
$$\left(0:_{\mathscr{R}(M)}L\right)=\bigoplus_{n\geq0}\left(M^{n+1}:_{\mathscr{R}(M)}y\right)\cap M^n.$$


Consider the natural projection $\pi$ from $\mathscr{R}(M)$ to $\mathscr{F}(M)$. Then
\begin{equation}\label{eqteoralgcoefestrutura3}\pi\left(\oplus_{n\geq0}\left(M^{n+1}:_{\mathscr{R}(M)}y\right)\cap M^n\right)=
\bigoplus_{n\geq0} \frac{\left(M^{n+1}:_{\mathscr{R}(M)}y\right)\cap
M^n+\mathfrak{m}M^n}{\mathfrak{m}M^n}=:H_k.
\end{equation} 

Then $H_k$ is an $R$-module and we have a surjective 
$R$-homomorphism
$$
\begin{array}{rcl}\displaystyle\frac{\mathscr{R}(M)}{(0:_{\mathscr{R}(M)}L)} & \longrightarrow &
\displaystyle\frac{\mathscr{F}(M)}{H_k}\end{array},
$$
and hence, by equation, (\ref{eqteoralgcoefestrutura2})
\begin{equation}\label{eqteoralgcoefestrutura4}\dim\left(\frac{\mathscr{F}(M)}{H_k}\right)\leq s-k.
\end{equation}
By assumption, we have $\dim(\mathscr{F}(M))=s$, that is, $s=s(M)$.  By equation (\ref{eqteoralgcoefestrutura4}) and by \cite[Lemma 2(E)]{shah}, we can obtain elements ${\overline{x}}_1,\ldots,{\overline{x}}_s$ in the fiber cone $\mathscr{F}(M)$ such that
\begin{enumerate}
\item[(1) ]{${\overline{x}}_1,\ldots,{\overline{x}}_s$ are homogeneous of the same degree $n_0\geq 1$;}
\item[(2) ]{$\ell_R\left(\frac{\mathscr{F}(M)}{\left({\overline{x}}_1,\ldots,{\overline{x}}_s\right)}\right)<\infty$;}
\item[(3) ]{We have $\dim\left(\frac{\mathscr{F}(M)}{\left({\overline{x}}_1,\ldots,{\overline{x}}_k\right)}\right)=s-k$;}
\item[(4) ]{${\overline{x}}_1,\ldots,{\overline{x}}_k$ are in
$H_k$.}
\end{enumerate}
Let ${{x}}_1,\ldots,{{x}}_s$ be any preimages of
${\overline{x}}_1,\ldots,{\overline{x}}_s$ in $M^{n_0}.$ By Lemma \ref{lemaferramentas}  and $(2)$ just above, the elements ${{x}}_1,\ldots,{{x}}_s$ form a minimal reduction of $M^{n_0}$. Since ${\overline{x}}_1,\ldots,{\overline{x}}_k$ are in $H_k$ we obtain by equation (\ref{eqteoralgcoefestrutura3}) that $y$ is in the projection $(M^{n_0+1}:_{\mathscr{R}(M)}x_1,\ldots,x_k)$. Therefore,
$M_{k}\subseteq\cup(M^{n+1}:_{\mathscr{R}(M)}(x_1,\ldots,x_k))$ and hence $M_{k}\subseteq\cup(M^{n+1}:_F(x_1,\ldots,x_k))\cap M^{sat}.$

Conversely, take $y\in(M^{n_0}:_Fx_1,\ldots,x_k)\cap
M^{sat}.$ In particular $yx_i\in M^{n_0}$. Thus, by Lemma
\ref{coropropoeliminandoelred2}, $y\in\overline{M}$. Therefore $M$ is a
reduction of $(M,y)$, i.e.,
${(M,y)}^{s+n}={\left(M,y\right)}^sM^n$ for some $s\geq0$ and all $n\geq0$. Then
$$
\begin{array}{lll}
\vspace{0.3cm}
\ell_R\left(\frac{{\left(M,y\right)}^{s+n}}{M^{n+s}}\right) & =
& \ell_R\left(\frac{{\left(M,y\right)}^sM^n}{M^{n+s}}\right) \\
\vspace{0.3cm} & =
& \displaystyle\sum^s_{i=1}\ell_R\left(\frac{{\left(M,y\right)}^iM^{n+s-i}}{{\left(M,y\right)}^{i-1}M^{n+s-i-1}}\right) \\
\vspace{0.3cm} & = &
\displaystyle\sum^s_{i=1}\ell_R\left(\frac{{\left(M,y\right)}^{i-1}M^{s-i}{\left(M,y\right)}M^n}
{{\left(M,y\right)}^{i-1}M^{s-i}M^{n+1}}\right) \\
\vspace{0.3cm} & \leq &
\displaystyle\sum^s_{i=1}c_i\ell_R\left(\frac{\left(M,y\right)M^n}{M^{n+1}}\right)
\end{array}
$$
for some positive constant integers $c_i$. The last inequality can be seen by canonically mapping for each $i$:
$$
\bigoplus^{c_i}_{n=1}L_n\longrightarrow\frac{{(M,y)}^{i-1}M^{s-i}{(M,y)}M^n}{{(M,y)}^{i-1}M^{s-i}M^{n+1}}
\longrightarrow 0,
$$ where $L_n=(M,y)M^n/M^{n+1}$. Choose $c_i$ as
the minimal number of generators of ${(M,y)}^{i-1}M^{s-i}$. Set
$c=\sum c_i$. Thus
\begin{equation}\label{bound} \ell_R\left(\frac{{\left(M,y\right)}^{s+n}}{M^{n+s}}\right)\leq
c\cdot\ell_R(L_n).
\end{equation}

Since $y\in M^{sat}$ there exists a positive integer $p$ such that $y{\mathfrak m}^p\subseteq M.$ Set $L=\bigoplus_{n\geq0}L_n$ and $G=\mathscr{R}(M)/{\mathfrak m}^p\mathscr{R}(M)=\bigoplus_{n\geq0} M^n/{\mathfrak m}^pM^n.$ Then, since $y{\mathfrak m}^p\subseteq M,$ we have ${\mathfrak m}^p\mathscr{R}(M)\subseteq \mbox{Ann}_{\mathscr{R}(M)}(L).$ Hence $L$, which is an $\mathscr{R}(M)$-module, has a natural structure of an $G$-graded module. Thus $\ell_R(L_n)$ is, for large $n$, a polynomial in $n$ of degree at most $\dim(L)-1.$ But $\dim(L)=\dim(G/\mbox{Ann}_G(L)).$ Notice that
$$
\mbox{Ann}_G(L)=\left(0:_{\frac{\mathscr{R}(M)}{{\mathfrak m}^p\mathscr{R}(M)}}\oplus_{n\geq0}\frac{\left(M,y\right)M^n}{M^{n+1}}\right)=\bigoplus_{n\geq0}
\frac{\left(M^{n+1}:_{\mathscr{R}(M)}y\right)\cap M^n}{\mathfrak{m}^pM^n}=:A_k.
$$

Since $y\in(M^{n_0}:_Fx_1,\ldots,x_k)$ it follows that $x'_1,\ldots,x'_k\in A_k,$ where $x'_i$ is the image of $x_i$ in $G$. Thus, since $\dim(G)=\dim\left(\mathscr{F}(M)\right)=s$, one concludes that
$$
\dim\left(\frac{G}{\mbox{Ann}_G(L)}\right)\leq \dim\left(\frac{G}{\left(x'_1,\ldots,x'_k\right)}\right)=s-k.
$$

Therefore, by the inequality (\ref{bound}), $\ell_R({(M,y)}^{s+n}/M^{s+n})$ is bounded by a polynomial in  $n$ of degree $s-(k+1)$ for all large $n$. Thus
$$
\deg\left(P_{\frac{(M,y)}{M}}(n)\right)<s-k,
$$
which implies that $M\subset\left(M,y\right)\subset M^F_{k}$. Therefore $y\in M^F_{k}$.
\end{proof}

\begin{theorem}\label{teorcoefestruturafinaparmod}{\rm [Finer Structure Theorem for Coefficient Modules].} Let $(R,\mathfrak{m})$ be a formally equidimensional and $d$-dimensional local ring with infinite residue field. Let $M$ be a finitely generated $R$-submodule of $F$ with analytic spread $s:=s(M)\geq 1$. Then there exists $n_0\geq1$ and a minimal reduction
$x_1,\ldots,x_s$ of
$M^{n_0}$ such that
$$
M^F_{k}=\left(M^{n_0+1}:_F\left(x_1,\ldots,x_k\right)\right)\cap M^{sat}, \text{ for all } k=1,\dots,s. 
$$
\end{theorem}

{\it Proof:} We know, by definition of $M^F_{k}$ and for all large $n$, that
$$
\ell_R\left(\frac{M_{k}^{n}}{M^n}\right)=P_{\frac{M_{k}}{M}}(n)\;\mbox{ and
that } \;\deg\left(P_{\frac{M_{k}}{M}}(n)\right)<s-k.
$$

In particular, since $M\subseteq M^F_{k}$, we  also
have
\begin{equation}\label{eqn-finer1}\ell_R\left(\frac{M^F_{k}M^{n-1}}{M^n}\right)\leq
P_{\frac{M_{k}^F}{M}}(n).
\end{equation}

Set $L^{(k)}:=\bigoplus_{n\geq 0}(M^F_{k}M^{n})/M^{n+1}.$ Clearly $L^{(k)}$ has the structure of an $\mathscr{R}(M)$-module and, by (\ref{eqn-finer1}), we have
\begin{equation}\label{eqn-finer2}\dim\left(L^{(k)}\right)=\dim\left(\frac{\mathscr{R}(M)}{\left(0:_{\mathscr{R}(M)}L^{(k)}\right)}\right)\leq
s-k.
\end{equation} We can easily prove that
$$
\left(0:_{\mathscr{R}(M)}L^{(k)}\right)=\bigoplus_{n\geq0}\left(M^{n+1}:M^F_{k}\right)\cap M^n.
$$

Consider the natural projection $\pi$ from $\mathscr{R}(M)$ to $\mathscr{F}(M)$. Then $$\pi\left(\oplus_{n\geq 0}\left(M^{n+1}:M_{k}\right)\cap
M^n\right)= \bigoplus_{n\geq1} \frac{\left(M^{n+1}:M^F_{k}\right)\cap
M^n+\mathfrak{m}M^n}{\mathfrak{m}M^n}=:H_k.$$ We have a
surjective $R$-homomorphism
$$\begin{array}{rcl}\displaystyle\frac{\mathscr{R}(M)}{\left(0:_{\mathscr{R}(M)}L^{(k)}\right)} & \longrightarrow &
\displaystyle\frac{\mathscr{F}(M)}{H_k}\end{array},$$ 
and hence, by equation (\ref{eqn-finer2}), $\dim(
\mathscr{F}(M)/H_k)\leq s-k,$ for all $k=1,\dots, s.$ Further, $H_1\subseteq \cdots\subseteq H_s$ and each $H_k$ is a homogeneous ideal of $\mathscr{F}(M).$ Thus, by \cite[Lemma 2(F)]{shah}, we can choose homogeneous elements
${\overline{x}}_1,\ldots,{\overline{x}}_s$ in $\mathscr{F}(M)$ of some fixed degree $n_0\geq1$ such that
$$
\ell_R\left(\frac{\mathscr{F}(M)}{\left({\overline{x}}_1,\ldots,{\overline{x}}_s\right)}\right)<\infty
$$
and ${\overline{x}}_1,\ldots,{\overline{x}}_k$ are in $H_k$, for all $1\leq k\leq s$. Let $x_1,\ldots,x_s$ be any preimages of ${\overline{x}}_1,\ldots,{\overline{x}}_s$ in $M^{n_0}.$ Then, by Lemma \ref{lemaferramentas}, $x_1,\ldots,x_s$ form a minimal reduction of $M^{n_0}.$ Now, since ${\overline{x}}_1,\ldots,{\overline{x}}_k$ are in $H_k$ we have that $M^F_{k}\subseteq (M^{n_0+1}:_Fx_1,\ldots,x_k).$ Therefore $M^F_{k}\subseteq
\left(M^{n_0+1}:_Fx_1,\ldots,x_k\right)\cap M^{sat}$ for $1\leq k\leq s$.

The inclusion $(M^{n_0+1}:_Fx_1,\ldots,x_k)\cap
M^{sat}\subseteq M^F_{k}$ follows straight from Theorem \ref{teorcoefestruturaparmod}.

\section{The module coefficients between \textit{I}(\textit{M})\textit{M} and \textit{M}}\label{CoefAss}

In this section, our main focus is on demonstrating the existence of coefficient modules between $I(M)M$ and $M$. We also give the structural properties of these coefficient modules. To this end, we begin by presenting some preliminary facts and concepts.
In \cite{Katz2}, Katz and Kodiyalam investigate the generalization for the associated graded ring generalized $G_{I(M)}(\mathscr{R}(M))$.

Next lemma is a generalization of \cite[Lemma 2(A)]{shah}.

\begin{lemma} \label{equ} Let $M$ be a finitely generated $R$-submodule of $F$ with $\ell_R(F/M)<\infty$ and let $a_1,\ldots,a_s$ be elements of $M^{n_0}$ for some fixed $n_0>0$. We denote $a_i':=a_i+I(M)M^{n_0}$ for $i=1,\ldots,s$ and some $s\geq 1$. Then
$$\ell_{R}\left(\frac{G_{I(M)}(\mathscr{R}(M))}{\left(a_1',\ldots,a_s'\right)}\right)<\infty
\Longleftrightarrow\ 
(a_1,\ldots,a_s)\ \mbox{is a reduction of}\ M^{n_0}.$$
\end{lemma}

\begin{proof} A well-known property of Fitting ideals says that $\sqrt{I(M)}=\sqrt{\mbox{Ann}_R(F/M)}$, so $I(M)$ is an $\mathfrak{m}$-primary ideal.
Firstly, observe that
\begin{equation}\label{equi}
\ell_{R}\left(\frac{G_{I(M)}(\mathscr{R}(M))}{\left(a_1',\ldots,a_s'\right)}\right)<\infty
\Leftrightarrow\ 
(x_1',\dots,x_s')\frac{M^{n_0n}}{I(M)M^{n_0 n}}=\frac{M^{n_0+n_0n}}{I(M)M^{n_0+n_0 n}},
\end{equation}
 for $n\gg0$.

$(\Rightarrow)$ Since for $n\gg0$,
$$(x_1',\dots,x_s')\frac{M^{n_0n}}{I(M)M^{n_0 n}}=\frac{M^{n_0+n_0n}}{I(M)M^{n_0+n_0 n}},$$  
we get $(x_1,\dots,x_s)M^{n_0n}+I(M)M^{n_0 +n_0n}=M^{n_0+n_0n}.$
And then, by Nakayama's Lemma one concludes that $(x_1,\dots,x_s)M^{n_0n}=M^{n_0+n_0n}$, which means $(x_1,\dots,x_s)$ is a reduction of $M^{n_0}$.

$(\Leftarrow)$ If $(x_1,\dots,x_s)(M^{n_0})^n=(M^{n_0})^{n+1}$ for some $n\geq 1$ then $(x_1,\dots,x_s)M^{n_0n}=M^{n_0+n_0n}$ for $n\gg0$, in a way that $(x_1,\dots,x_s)M^{n_0n}+I(M)M^{n_0 +n_0n}=M^{n_0+n_0n}.$ The result follows by Equation \eqref{equi}.
\end{proof}

Let $M$ be a finitely generated $R$-submodule of $F$ with $\ell_R(F/M)<\infty$. Then $\sqrt{I(M)}=\mathfrak{m}$, which implies 
$$\dim\left(G_{I(M)}\left(\mathscr{R}(M)\right)\right)= \dim \left(\frac{\mathscr{R}(M)}{I(M)\mathscr{R}(M)}\right)= \dim \left(\frac{\mathscr{R}(M)}{\mathfrak{m}\mathscr{R}(M)}\right)=\dim\left(\mathscr{F}(M)\right)=s.$$ 

Consider an $R$-submodule $L$ of $M$ such that $I(M)M \subseteq L \subseteq M$, and let $\mathscr{P}=\bigoplus_{n\geq 1}\mathscr{P}_n$ with $\mathscr{P}_n=\frac{LM^{n-1}}{I(M)M^n}$ be a $G_{I(M)}(\mathscr{R}(M))$-submodule.

According to \cite[Theorem 4.1.3]{BH}, the numerical function $n \mapsto \ell_R(\mathscr{P}_n)$ becomes a polynomial of degree $\dim \mathscr{P} - 1$, for sufficiently large $n$. In other words, it is a polynomial of degree less  than or equal to $s-1$.

\begin{lemma}  \cite[Lemma 2.1]{LP} \label{ttt}
 Let $R$ be a Noetherian ring and $M$ be an $R$-module, and let $I$, $J$, and $L$ be three $R$-submodules of $M$ such that $L \subseteq I$, $L \subseteq J$, $\ell_R(I/L)<\infty$ and $\ell_R(J/L)<\infty$. Then
$$
\ell_R\left(\frac{I+J}{L}\right) \leq \ell_R\left(\frac{I}{L}\right) + \ell_R\left(\frac{J}{L}\right).
$$
\end{lemma}

By convention, a polynomial of degree less than $-1$ is the zero polynomial.

\begin{theorem}
Let $(R,\mathfrak{m})$ be a $d$-dimensional Noetherian local ring with infinite residue field and let $M$ be a finitely generated $R$-submodule of $F$ with $\ell_R(F/M)<\infty$ and analytic spread $s=s(M)\geq 1$. Then, for each $k=1,\dots,s$, there exists an unique largest module $I(M)M \subseteq M_{[k]}\subseteq M$ such that 
$${\rm deg }\left[\ell_R\left(\frac{M_{[k]}M^{n-1}}{I(M)M^n}\right)\right]\leq s-(k+1),
$$
\noindent for $n\gg 0$. Moreover, we have
$$
I(M)M \subseteq M_{[s]} \subseteq M_{[s-1]} \subseteq \cdots \subseteq M_{[1]} \subseteq M.
$$
\end{theorem}

\vspace{0.3cm}

\begin{proof}
Fixed $k\in \{1,\dots,s\}$, define the set
$$
\begin{array}{lll}
\displaystyle V_{k} & = & \displaystyle \{ L  \subset F \ | \ L \text{ is an $R$-submodule of } F , \ M \supseteq L \supseteq I(M)M 
 \text{ and }  \\ 
 \\
 & & \displaystyle \hspace{4cm} \deg\left[\ell_R\left(\frac{LM^{n-1}}{I(M)M^n} \right)\right]\leq s-(k+1) \text{ for } n\gg0 \}.
\end{array}
$$

The set $V_k$ is not empty because $I(M)M\in V_k$. As $F$ is a Noetherian $R$-module, $V_k$ contains a maximal element, say $J$. Let $L\in V_k$. In order to prove that $J$ is the unique maximal element in $V_k$, we shall prove that $L\subseteq J$. Let $x\in L$. Then $(I(M)M,x)\subseteq L$ so that, 
$$
\begin{array}{lll}
\displaystyle \ell_R\left(\frac{xM^{n-1}+I(M)M^n}{I(M)M^n}\right) & = & \displaystyle \ell_R\left(\frac{\left(I(M)M,x\right)M^{n-1}}{I(M)M^n}\right) \\
\\
& \leq & \displaystyle \ell_R\left(\frac{LM^{n-1}}{I(M)M^n}\right).
\end{array}
$$ 

Then, by \cite[Theorem 4.1.3]{BH} and  for $n\gg 0$, we get
$$\deg\left[\ell_R\left(\frac{xM^{n-1}+I(M)M^n}{I(M)M^n}\right)\right]\leq s-(k+1).$$
By Lemma \ref{ttt}, we obtain the inequality
$$
\begin{array}{lll}
\displaystyle \ell_R\left(\frac{\left(J,x\right)M^{n-1}}{I(M)M^n}\right)
& = & \displaystyle \ell_R\left(\frac{(J,x)M^{n-1}+I(M)M^n}{I(M)M^n}\right) \\
\\
& = & \displaystyle \ell_R\left(\frac{JM^{n-1}+xM^{n-1}+\mathfrak{m}I^n}{I(M)M^n}\right) \\
\\
 &\leq &
\displaystyle \ell_R
\left(\frac{JM^{n-1}}{I(M)M^n}\right)+\ell_R\left(\frac{xM^{n-1}+I(M)M^n}{I(M)M^n}\right).
\end{array}
$$
By \cite[Theorem 4.1.3]{BH}, we conclude that 
$\deg[\ell_R((J,x)M^{n-1}/I(M)M^n)]\leq d-(k+1)$, so that
$(J,x)\in V_k$. Since $J$ is maximal in $V_k$, it follows that $x\in J$. 
We denote this unique maximal element $J$ by $M_{[k]}$.
Moreover, for $k=1,\dots,s-1$, since 
$$\deg\left[\ell_R\left(\frac{M_{[k+1]}M^{n-1}}{I(M)M^n}\right)\right]\leq s-(k+2)< s-(k+1),
$$
for sufficiently large $n>0$ and by maximality of $M_{[k]}$, one has $M_{[k+1]}\subseteq M_{[k]}$.
\end{proof}

\begin{definition} The module $M_{[k]}$, for each $k=1,\dots, k\leq s$, obtained in the previous result will be called $k$-th Coefficient Module of the Associated Graded Ring $G_{I(M)}(\mathscr{R}(M))$. 
\end{definition}

\begin{theorem}\label{structure1}
Let $(R,\mathfrak{m})$ be $d$-dimensional Noetherian local ring with infinite residue field and let $M$ be a finitely generated $R$-submodule of $F$ with $\ell_R(F/M)<\infty$ and analytic spread $s=s(M)\geq 1$. Let $k\in\{1,\dots,s\}$. Then
$$
M_{[k]}= \bigcup_{n_0 , \ \underline{x} }  M \cap \left(I(M)M^{n_0+1}:_Fx_1,\dots,x_k\right),
$$
where $n_0\geq 1$ and $\overline{x}=(x_1,\dots,x_k, \dots, x_s)$ is a minimal reduction of $M^{n_0}$.
\end{theorem}

\begin{proof}
Let $y\in M_{[k]}\subseteq M$. We will prove that $y\in M \cap \left(I(M)M^{n_0+1}:_Fx_1,\dots,x_k\right)$ for some $n_0\geq 1$ and some minimal reduction  $(x_1,\dots,x_k, \dots, x_s)$ of $M^{n_0}$, which proves one inclusion.
From the inequality
$$
\ell_R\left(\frac{yM^{n-1}+I(M)M^n}{I(M)M^n} \right)= \ell_R\left(\frac{(I(M)M,y)M^{n-1}}{I(M)M^n} \right)\leq \ell_R\left(\frac{M_{[k]}M^{n-1}}{I(M)M^n} \right),
$$
we obtain by \cite[Theorem 4.1.3]{BH} that
$$
\deg\left[\ell_R\left(\frac{yM^{n-1}+I(M)M^n}{I(M)M^n} \right)\right] \leq \deg\left[\ell_R\left(\frac{M_{[k]}M^{n-1}}{I(M)M^n} \right)\right] \leq s-(k+1),
$$
for $n\gg 0$. Now, define a graded $G_{I(M)}(\mathscr{R}(M))$-submodule $$E=\frac{\left(I(M)M,y\right)}{I(M)M}G_{I(M)}(\mathscr{R}(M))=\bigoplus_{n\geq 1}\frac{\left(I(M)M,y\right)M^{n-1}}{I(M)M^n}=\bigoplus_{n\geq 1}\frac{yM^{n-1}+I(M)M^n}{I(M)M^n}$$ of $G_{I(M)}(\mathscr{R}(M))$. Set $A=\text{Ann}_R(E)$. It follows by \cite[Theorem 4.1.3]{BH} that 
$$\deg\left[\ell_R\left(\frac{yM^{n-1}+I(M)M^n}{I(M)M^n} \right)\right]=\dim(E)-1.$$ Thus
$\dim E - 1\leq s-(k+1)$, that is, $\dim E=\dim G_{I(M)}(\mathscr{R}(M))/A\leq s-k$. Note that $A$ is a homogeneous ideal in $G_{I(M)}(\mathscr{R}(M))$, so by \cite[Lemma 2(E)]{shah} we get homogeneous elements $x_1',\dots,x_s'$ in $G_{I(M)}(\mathscr{R}(M))$ of the same degree, say $n_0\geq 1$, such that 
$$\ell_R\left(\frac{G_{I(M)}(\mathscr{R}(M))}{\left(x_1',\dots,x_s'\right)} \right)<\infty \text{ and } x_1',\dots,x_k'\in A.$$
Let $x_i$ denote a preimage of $x_i^{\circ}$
in $M^{n_0}$, for $1\leq i\leq d$. Then, by Lemma \ref{equ} the ideal $(x_1,\dots,x_s)$ is a minimal reduction of $M^{n_0}$. Also, $x_1'E=\cdots =x_k'E=0$, and since $y+I(M)M\in E$ we obtain
$0=x_i'(y+I(M)M)=x_iy+I(M)M^{n_0+1}$ for each $i=1,\dots,k$,
that is, $y\in M\cap (I(M)M^{n_0+1}:x_1,\dots,x_k)$.  

For the reverse inclusion, suppose $y\in M\cap (I(M)M^{n_0+1}:x_1,\dots,x_k)$, where $(x_1,\dots,x_k,\dots,x_s)$ is a minimal reduction of $M^{n_0}$. By Lemma \ref{equ}, the canonical images $x_1',\dots,x_s'$ in $G_{I(M)}(\mathscr{R}(M))$ form a system of parameters of $G_{I(M)}(\mathscr{R}(M))$. Now, we prove that
$$
\deg\left[\ell_R\left(\frac{\left(I(M)M,y\right)M^{n-1}}{I(M)M^n} \right)\right]\leq s-(k+1), \text{ for }n\gg 0.
$$
Consider the graded $G_{I(M)}(\mathscr{R}(M))$-submodule $$E=\frac{\left(I(M)M,y\right)}{I(M)M}G_{I(M)}(\mathscr{R}(M))=\bigoplus_{n\geq 1}\frac{\left(I(M)M,y\right)M^{n-1}}{I(M)M^n}$$ of $G_{I(M)}(\mathscr{R}(M))$. By \cite[Theorem 4.1.3]{BH}, for $n$ large, we have
$$\deg\left[\ell_R\left(\frac{(I(M)M,y)M^{n-1}}{I(M)I^n} \right)\right] = \dim(E) -1=\dim\left(\frac{\mathscr{F}(I)}{\text{Ann}_R(E)}\right)-1.$$
As
$y\in M\cap (I(M)M^{n_0+1}:x_1,\dots,x_k)$, we obtain that $x_{i}'y'=0$ in $G
_{I(M)}(\mathscr{R}(M))$, where $y'=y+I(M)M$, for each $i\in \{1,\dots,k\}$. Since $E=y'G_{I(M)}(\mathscr{R}(M))$, it follows that 
$x_1',\dots,x_k'\in \text{Ann}_R(E)=A$. Then $\dim(E)=\dim(G_{I(M)}(\mathscr{R}(M))/A)\leq \dim(G_{I(M)}(\mathscr{R}(M))/(x_1',\dots,x_k'))=s-k$, as $x_1',\dots,x_s'$ is a system of parameters of $G_{I(M)}(\mathscr{R}(M))$. Therefore we get that
$\deg[\ell_R((I(M)M,y)M^{n-1}/I(M)M^n)] \leq s-(k+1)$. From this, by maximality of $M_{[k]}$, we get $(I(M)M,y) \subseteq M_{[k]}$ which proves that $M\cap (I(M)M^{n_0+1}:x_1,\dots,x_k)\subseteq M_{[k]}$.
\end{proof}

\begin{theorem}\label{equim}
Let $(R,\mathfrak{m})$ be a formally equidimensional Noetherian and $d$-dimensional local ring with infinite residue field and $d\geq 1$. Let $M$ be a finitely generated $R$-submodule of $F$ such that $\ell_R(F/M)<\infty$ and assume that $s(M)=d+p-1$. Then $G_{I(M)}(\mathscr{R}(M))$ is equidimensional.
\end{theorem}
\begin{proof}
We have $\dim(G_{I(M)}(\mathscr{R}(M)))=\dim(\mathscr{R}(M)/I(M)\mathscr{R}(M))=s(M)=d+p-1$. Let $Q+I(M)\mathscr{R}(M)$ be a minimal prime ideal of $\mathscr{R}(M)/I(M)\mathscr{R}(M)=G_{I(M)}(\mathscr{R}(M))$. We need to prove that $\dim(\mathscr{R}(M)/Q)=\dim(G_{I(M)}(\mathscr{R}(M)))=d+p-1$. Assuming $R$ is a domain by \cite[Lemma 16.2.2 (2)(4)]{huneke}, and since $R$ is formally equidimensional, it follows that $R$ is universally catenary, and thus $\mathscr{R}(M)$ is catenary. Consequently, we have $\text{ht}(I(M)\mathscr{R}(M))+\dim(\mathscr{R}(M)/I(M)\mathscr{R}(M))=\dim(\mathscr{R}(M))=d+p$ (see \cite[Lemma 16.2.2 (4)]{huneke}).

Since $\dim(G_{I(M)}(\mathscr{R}(M)))=d+p-1$, we obtain $\text{ht}(I(M)\mathscr{R}(M))=1$. Let $Q'$ be a minimal prime over $I(M)\mathscr{R}(M)$ such that $\text{ht}(I(M)\mathscr{R}(M))=\text{ht}(Q')$. For an arbitrary prime ideal $Q+I(M)\mathscr{R}(M)$ of $\mathscr{R}(M)/I(M)\mathscr{R}(M)$, we have $\text{ht}(Q+I(M)\mathscr{R}(M))\geq \text{ht}(Q/Q')=\text{ht}(Q)-\text{ht}(Q')$. Thus, $\text{ht}(Q+I(M)\mathscr{R}(M)) = \text{ht}(Q)-1$. In particular, if $Q+I(M)\mathscr{R}(M)$ is minimal, then $\text{ht}(Q)=1$. Since $\text{ht}(Q)+\dim(\mathscr{R}(M)/Q)=\dim(\mathscr{R}(M))=d+p$, we conclude that $\dim(\mathscr{R}(M)/Q)=d+p-1$, as desired.
\end{proof}

\section{Applications}
The purpose of this section is to provide some applications or properties of the coefficient modules, obtained in Sections \ref{CoeffMod} and \ref{CoefAss}, in relation to the Ratliff-Rush module and the relative Buchsbaum-Rim coefficient. Furthermore, we also present a generalization of a result obtained by Shah in \cite[Theorem 4]{shah}.

\subsection{Applications of coefficient modules between $M$ and $\overline{M}$}

We recall the definition of Ratliff-Rush closure modules given in \cite[Definition 2.2]{jung} or \cite[Definition 1.1]{Endo}.

\begin{definition}{\rm (\cite[Definition 1.1]{Endo})} Let $R$ be a Noetherian ring, $M$ a finitely generated $R$-submodule of  $F$. 
 We define the Ratliff-Rush Closure Module $\widetilde{M}$ of $M$ as the following $R$-submodule of $F$
$$\widetilde{M}:=\bigcup_{n\geq0}\left(M^{n+1}:_{F}M^n\right).$$
\end{definition}

The next result relates the coefficient module $M_{s}$ of $M$ with its Ratliff-Rush closure $\widetilde{M},$ extending to arbitrary ideals a similar result proved by Shah in \cite[Corollary 1(E)]{shah} for ${\mathfrak m}$-primary ideals. 

Here, we consider the Ratliff-Rush closure module of $M$ to be the module $\widetilde{M}$
such that $({\widetilde{M}})^n=M^n$ for all large $n$, maximal in $F$ with this 
property.

\begin{proposition} Let $(R,\mathfrak{m})$ be a formally equidimensional Noetherian and $d$-dimensional local ring with infinite residue field. Let $M$ be a finitely generated $R$-submodule of $F$ and  with analytic spread $s:=s(M)\geq 1$. Then
$$
M^F_{s}=\widetilde{M}\cap M^{sat}.
$$
\end{proposition}

\begin{proof} Set $N=\widetilde{M}\cap M^{sat}.$ Since $M\subseteq N\subseteq \widetilde{M}$ and $({\widetilde{M}})^n=M^n$ for all large $n$, we have that $N^n=M^n$, for all large $n.$ Clearly $N\subseteq q(M)$ and, since $\ell_R(N^n/M^n)=0$ for all large $n$, we get $N\subseteq M_{s}.$

Conversely, since $\deg(P_{M_{s}/M}(n))<0$ we have that $\ell_R((M^F_{[s]})^n/M^n)=0$ for all large $n$. Thus, $M^F_{s}\subseteq \widetilde{M}$ and therefore $M^F_{s}\subseteq \widetilde{M}\cap M^{sat}.$
\end{proof}

\begin{corollary}
Let $(R,\mathfrak{m})$ be a formally equidimensional and $d$-dimensional local domain with infinite residue field. Let $M$ be a finitely generated $R$-submodule of $F$ with analytic spread $s:=s(M)$. Then 
$$M_{s}^F=M^n\cap M^{sat},$$
for all large $n$ and $M^F_{s}$ is maximal in $F$ with respect to this property.
\end{corollary}

\begin{proof} Since $M\subseteq M^F_{s}$, we have 
$M^n\subseteq \left(M^F_{s}\right)^n$. Since $M_{s}^F=\widetilde{M}\cap M^{sat}$, observe that
$$\left(M_{s}^F\right)^n=\left(\widetilde{M}\cap M^{sat}\right)^n\subseteq \left(\widetilde{M}\right)^n\cap \left( M^{sat}\right)^n\subseteq \left(\widetilde{M}\right)^n\cap M^{sat}=M^n\cap M^{sat}.$$
Now maximality follows from Theorem \ref{teoralgcoefarb}.
\end{proof}

For $0 \leq k \leq s$, we set 
\[
V_{k}(M)=\left\{ L\subseteq F|\ M\subseteq L,\ 
\ell_R\left(\frac{L}{M}\right)<\infty,\ \mbox{deg}\left(P_{\frac{L}{M}}(n)\right)<s-k\right\}.
\]

By Theorem \ref{teoralgcoefarb}, we establish the existence of a unique maximal element $M_k$ in each $V_k(M)$, which is referred to as the $k$th coefficient module of $M$. In the following, we give a result that generalize \cite[Proposition 3.16]{Endo}. This result show that $\widetilde{M}$ is the largest $R$-submodule $L$ such  that satisfies the properties $\ell_R(L/M)<\infty$ and $\mbox{deg}(P_{L/M}(n))<s-k.$

\begin{proposition}
Suppose that $M$ is a faithful $R$-module. Then $\widetilde{M}\in V_{}$ and  $\widetilde{M} \subseteq L$ for every $L\in V_{k}(M)$
\end{proposition}

\begin{proof}
Note that $\widetilde{M}\neq F$  by \cite[Lemma 3.15]{Endo}. Choose an integer $r>0$ such that $(\widetilde{M})^n=M^n$ for every $n\geq r$, by \cite[Proposition 3.13 (1)]{Endo}. Then we have $M^n\subseteq (\widetilde{M})^n \subseteq \widetilde{M^n}\subseteq M^n$, so that $M^n=(\widetilde{M})^n$. Thus $\ell_R(L/\widetilde{M})=\ell_R(L/M)<\infty$ and $\mbox{deg}(P_{L/\widetilde{M}}(n))=\mbox{deg}(P_{L/M}(n))<s-k$  for every $n \gg 0,$ as desired.
\end{proof}

Let $L$ and $M$ two $R$-submodules of $F$ such that $L\subset M$, Set $$W(L): = \frac{\mathscr{R}(L)}{\mathscr{R}(M)}=\bigoplus_{n\geq 0}\frac{L^n}{M^n}.$$ 

Next, we give a result that provides an estimation of the Krull dimension of the $R$-algebra $W(L)$

\begin{proposition}\label{Dim1} Let $(R,\mathfrak{m})$ be a $d$-dimensional Noetherian local ring. Let $M$ and $L$ two $R$-submodules of $F$ such that $M$ is a reduction of $L$ and $\ell_R(L/M)<\infty$. Then $\dim(W(L))\leq d+p-1$.
\end{proposition}
\begin{proof} Set $W = W(L)$. Since  $L\subseteq \overline{M}$, one concludes that $\mathscr{R}(L)$ and $W$ are finitely generated $\mathscr{R}(M)$-module ( \cite[Theorem 16.2.3]{huneke}). Furthermore, $W_n$ has finite length for all $n$. Then, by \cite[Theorem 2.1.5]{Roberts} or \cite[Theorem 16.5.6(3)]{huneke}, for sufficiently large $n$, $\ell_R(W_n)$ is a polynomial function of degree at most $d+p-2.,$ so that $\dim(W)\leq d+p-1$, by \cite[Proposition 4.6]{herzog}.  
\end{proof}

Next, we provide a result similar to Proposition \ref{Dim1} but in the case where $M$ has finite colength and a maximal analytic spread of $d+p-1$. For $0 \leq k \leq d+p-1$, we define
$$B_{k}(M):=\left\{ L\subseteq F \ |\ \ M\subseteq L, \
{\rm e}_i^F(M)={\rm e}_i^F(L), \ 0\leq i\leq k\right\}.$$
It is worth recalling that in \cite{Perez-Ferrari}, it was showed that each $B_{k}(M)$ contains an unique maximal element $M_{\{k\}}^F$, which is also called  the $k$-th Coefficient Module of $M$.

\begin{proposition}
Let $(R,\mathfrak{m})$ be a $d$-dimensional Noetherian local ring. Let $k \in \{0,\dots, d+p-1\}$ and $L\in V_{k}(M)$. Then $\dim(W(L))\leq d+p-1-k$.    
\end{proposition}
\begin{proof}
Set $W = W(L)$. From the short exact sequence

$$
0 \longrightarrow W(+1)\longrightarrow \bigoplus_{n\geq 0}\frac{F^{n+1}}{M^{n+1}}\longrightarrow\bigoplus_{n\geq 0}\frac{F^{n+1}}{L^{n+1}}
\longrightarrow 0,
$$
for $n\gg 0$ we get
\[
\begin{array}{lll}
   \ell_R(W_{n+1}) & = & \ell_R\left(\frac{F^{n+1}}{M^{n+1}}\right)-\ell_R\left(\frac{F^{n+1}}{L^{n+1}}\right)  \\
   \\
     &=&\sum\limits_{i=0}^{d+p-1}{(-1)}^i{\rm e}_i^F(M)\left(\begin{array}{c}n+d+p-i-2\cr d+p-i-1\end{array}\right)\\
     \\
     & &-\sum\limits_{i=0}^{d+p-1}{(-1)}^i{\rm e}_i^F(L)\left(\begin{array}{c}n+d+p-i-2\cr d+p-i-1\end{array}\right)\\
     \\
     &=& \sum\limits_{i=0}^{d+p-1}{(-1)}^i\left[{\rm e}_i^F(M)-{\rm e}_i^F(L)\right]\left(\begin{array}{c}n+d+p-i-2\cr d+p-i-1\end{array}\right).
\end{array}
\]

As ${\rm e}_i^F(M)={\rm e}_i^F(L)$, for all $0\leq i\leq k$, we obtain, for $n\gg 0$, that $\ell_R(W_{n+1})$ is a polynomial of degree at most $d+p-2-k$. Hence, \cite[Proposition 4.6]{herzog},  $\dim(W(L)) \leq d+p-1-k.$

\end{proof}



Before presenting the next corollaries, let's review the following facts: Let $N \subseteq M \subseteq F$ be two $R$-modules such that $\ell_R(M/N) < \infty$. According to \cite[Theorem 16.5.6]{huneke}, for sufficiently large $n \in \mathbb{N}$, the numerical function $\ell_R(M^n/N^n)$ is a polynomial function denoted as $P_{M/N}(n)$ of degree at most $s-1$, where $s := d + \max\{{\rm rank}(M \otimes_R k(\mathfrak{p})):\ \mathfrak{p} \in {\rm Min}(R)\}$. This polynomial $P_{M/N}(n)$ can be expressed in the form:
$$P_{M/N}(n) = \sum_{i=0}^{s-1} (-1)^i {\rm e}_i^F(M,N)\binom{n+s-i-2}{s-i-1},$$
where ${\rm e}_i^F(N,M)$ is referred to as the $i$-th relative Buchsbaum-Rim coefficient of $(N,M)$.

Assuming $(R,\mathfrak{m})$ is a formally equidimensional ring and $N \subseteq M \subseteq F$ are two $R$-modules with $\ell_R(M/N) < \infty$, such that ${\rm ht}({\rm Ann}_R(F/M)) > 0$, the equivalence established by \cite[Theorem 16.5.6(3), Corollary 16.5.7]{huneke} states the following: $N$ is a reduction of $M$ if and only if $\mbox{deg}(P_{M/N}(n)) < d+p-1$ if and only if ${\rm e}_0^F(M,N) = 0$. Therefore, utilizing this equivalence, we can readily deduce that $\mbox{deg}(P_{M/N}(n)) < d+p-1-k$ if and only if ${\rm e}_i^F(M,N) = 0$ for $i = 0, 1, \dots , k$, for all $k = 0, 1, \dots, d+p-1$.

Now, in particular, we have ${\rm e}_i^F(M_{\{k\}}^F,M) = 0$ for $i = 0, \ldots, k$ if and only if $\mbox{deg}(P_{M_{\{k\}}^F/M}(n)) < d+p-1-k$. Thus, we obtain the following corollary of Theorem \ref{structure1}.

\begin{corollary} Let $(R,\mathfrak{m})$ be a formally equidimensional $d$-dimensional local ring with infinite residue field. Let $M$ be a finitely generated $R$-submodule of $F$ with $\ell_R(F/M)<\infty$. Then, for each $k=0,\ldots,d+p-1$, there exists an unique maximal $R$-submodule $M^F_{\lbrace k\rbrace}$ of $F$ containing $M$ such that
\begin{enumerate}
\item[(1) ]{${\rm e}_i^F(M)={\rm e}_i^F(M^F_{\lbrace k\rbrace})$ for $i=0,\ldots,k$;}
\item[(2) ]{$M\subseteq M^F_{\lbrace d+p-1\rbrace}\subseteq\cdots\subseteq M^F_{\lbrace 1\rbrace}\subseteq M^F_{\lbrace 0\rbrace}=\overline{M}$.}
\end{enumerate}
\end{corollary}

\subsection{The generalized associated graded ring} As an application of Section \ref{CoefAss}, we give a new result that shows a way to control the height of the associated prime ideals of $G_{I(M)}(\mathscr{R}(M))$, similar to the control established in \cite[Theorem 4]{shah}.

\begin{theorem}\label{height} Let $(R,\mathfrak{m})$ be a formally equidimensional $d$-dimensional local ring with infinite residue field.. Let $M$ be a finitely generated $R$-submodule of $F$ such that $\ell_R(F/M)<\infty$ and assume that $s(M)=d+p-1$. Let $k\in\{1,\dots,d+p-1\}$.
Then $I(M)M^n = (M^n)_{[k]}$ for every $n\geq 1$ if and
only if ${\rm ht}(P) <k$ for every $P\in {\rm Ass}(G_{I(M)}(\mathscr{R}(M)))$.    
\end{theorem}

\begin{proof}
Assume that $I(M)M^n = (M^n)_{[k]}$ for every $n\geq 1$ and take $P \in \text{Ass}(G_{I(M)}(\mathscr{R}(M)))$. Because $P$ is a graded ideal, we may write $P=(0':b')$ with $b \in M^{n} \hspace{0.05cm}\backslash \hspace{0.05cm} I(M)M^{n}$. Now suppose that $\text{ht}(P)\geq k$; then $\dim(G_{I(M)}(\mathscr{R}(M))/P)\leq (d+p-1) - k$. By \cite[Lemma 2(E)]{shah} there exist $z_1',\dots,z_{d+p-1}'\in G_{I(M)}(\mathscr{R}(M))$, with $z_i':=z_i+I(M)M^{n_0}$ for some $n_0\geq  1$ and every $i\in \{1,\dots,d+p-1\}$, such that
$$
\ell_{R}\left(\frac{G_{I(M)}(\mathscr{R}(M))}{\left(z_1',\ldots,z_{d+p-1}'\right)}\right)<\infty \text{ and } z_1',\ldots,z_k'\in P.
$$
Define $x_i=z_i^n$ for each $i$. So
$$
\ell_{R}\left(\frac{G_{I(M)}(\mathscr{R}(M))}{\left(x_1',\ldots,x_{d+p-1}'\right)}\right)<\infty \text{ and } x_1',\ldots,x_k'\in P,
$$
where $x_i':=x_i+I(M)M^{nn_0}$. One concludes that $b(x_1,\dots,x_k)\subset I(M)M^{n+nn_0}=I(M)(M^n)^{n_0+1}.$
By Theorem \ref{structure1} we obtain
$$
b\in M^n \cap (I(M)(M^n)^{n_0+1}:x_1,\dots,x_k)\subseteq (M^n)_{[k]},
$$
which is a contradiction.

Now, assume that $\text{ht}(P) <k$ for every $P\in \text{Ass}(G_{I(M)}(\mathscr{R}(M)))$ and fix $k\in \{1,\dots,d+p-1\}$. Let $n,n_0$ be positive integers and let $x_1,\dots,x_{d+p-1}$ be any minimal reduction of $M^{n_0}$. We must show that $I(M)M^n = (M^n)_{[k]}$. To do so, by Theorem \ref{structure1}, it suffices to show that
\begin{equation}\label{ig}
I(M)M^n = M^n \cap (I(M)(M^n)^{n_0+1}:x_1,\dots,x_k).
\end{equation}
In fact, letting $x_i':=x_i+I(M)M^{n_0}$ note that $\text{ht}(x_1',\dots,x_k')=k$ since $G_{I(M)}(\mathscr{R}(M))$ is equidimensional (by Theorem \ref{equim}). In this way, the ideal $(x_1',\dots,x_k')$ must contain a nonzero divisor. Take $y\in M^n \cap (I(M)(M^n)^{n_0+1}:x_1,\dots,x_k)$; then $y'x_i'=0'$ for each $i\in \{1,\dots,k\}$, where $y':=y+I(M)M^{n}$ and $x_i':=x_i+I(M)M^{nn_0}$. From this we derive \ref{ig}, as wished. 
\end{proof}

Next corollary improves \cite[Theorem 4]{shah}.

\begin{corollary}
Let $(R,\mathfrak{m})$ be a formally equidimensional $d$-dimensional local ring with infinite residue field and $d\geq 1$, let $I$ be an $\mathfrak{m}$-primary ideal of $R$ and let $k\in \{1,\dots,d\}$. Then the following are equivalents:
\begin{enumerate}
    \item $I^{n+1} = (I^n)_{[k]}$ for every $n\geq 1$; \vspace{0.2cm}
    \item $I^{n} = (I^n)_{k}$ for every $n\geq 1$; \vspace{0.2cm}
    \item ${\rm ht}(P)<k$ for every $P\in {\rm Ass}(G_I(R)).$
\end{enumerate}
\end{corollary}

\begin{proof}
It follows from Theorem \ref{height} and \cite[Theorem 4]{shah}.    
\end{proof}

\end{document}